\documentclass{amsart}
\usepackage{amsmath}
\usepackage{amssymb}
\usepackage{amsthm,amssymb,amsmath,enumerate,graphicx}

\newtheorem{theorem}{Theorem}[section]
\newtheorem{lemma}[theorem]{Lemma}
\newtheorem{assertion}[theorem]{Assertion}

\newtheorem{corollary}[theorem]{Corollary}

\theoremstyle{remark}

\newtheorem{definition}[theorem]{Definition}
\newtheorem{example}[theorem]{Example}

\newcommand{\cc}{\mathcal C}

\newcommand{\cl}{\mathcal L}
\newcommand{\cm}{\mathcal M}
\newcommand{\cn}{\mathcal N}

\newcommand{\cp}{\mathcal P}
\newcommand{\cq}{\mathcal Q}

\newcommand{\zn}{\mathbb{Z}_n}
\newcommand{\vc}{V^\circ}

\begin{document}

\makeatletter

\makeatother
\author{Ron Aharoni}
\address{Department of Mathematics\\ Technion, Haifa, Israel}
\thanks{The research of the first author was  supported by  an ISF grant, BSF grant no. 2006099 and by the Discount Bank Chair at the Technion.}
\email[Ron Aharoni]{raharoni@gmail.com}

\author{Dani Kotlar}
\address{Department of Computer Science\\ Tel-Hai college, Israel}

\author{Ran Ziv}
\address{Department of Computer Science\\ Tel-Hai college, Israel}

\title{Uniqueness of the extreme cases in theorems of Drisko and Erd\H{o}s-Ginzburg-Ziv}

\maketitle
\begin{abstract}
Drisko \cite{drisko}  proved (essentially) that every family of
$2n-1$ matchings of size $n$ in a bipartite graph  possesses a
partial rainbow matching of size $n$. In \cite{bgs} this was
generalized as follows: Any $\lfloor \frac{k+2}{k+1} n \rfloor
-(k+1)$ matchings of size $n$ in a bipartite graph have a rainbow
matching of size $n-k$. We extend this latter result to matchings of
not necessarily equal cardinalities.
 Settling a conjecture of
Drisko, we characterize those families of $2n-2$ matchings of size
$n$ in a bipartite graph that do not possess a rainbow matching of
size $n$. Combining this with  an idea of Alon \cite{alon}, we
re-prove a characterization of the extreme case in a well-known
theorem of Erd\H{o}s-Ginzburg-Ziv  in additive number theory.
\end{abstract}

\section{Introduction}
Let $\cc=(C_1, \ldots ,C_m)$ be a system of sets.  A set that is the
range  of an injective partial choice function from $\cc$ is called
a {\em rainbow set}, and is also said to be {\em multicolored} by
$\cc$.
 If $\phi$ is such a partial choice function and $C_i \in dom(\phi)$ we say
 that $C_i$ {\em colors} $\phi(C_i)$ in the rainbow set. If the elements of $C_i$ are sets then a rainbow set is said to be
a {\em (partial)
rainbow matching} if its range is a matching, namely it consists of disjoint sets.\\

Let $A$ be an $m \times n$ matrix whose elements are symbols. A {\em
partial transversal} in $A$ is a set  entries with no two in the
same row or column, and no two sharing the same symbol. If the
partial transversal is of size $\min(m,n)$ then it is called {\em
full}, or simply a {\em transversal}. A famous conjecture of
Ryser-Brualdi-Stein (see \cite{stein}) is that in an $n \times n$
Latin square there is a partial transversal of size $n-1$.
 Drisko \cite{drisko} showed that this is true if the assumption is ``doubled'':

\begin{theorem}\label{originaldrisko}
Let $A$ be an $m\times n$ matrix in which the symbols in each row
are all distinct. If $m\ge 2n-1$, then $A$ has a full transversal.
\end{theorem}

In \cite{ab} this was formulated in a slightly more general version:

\begin{theorem}\label{drisko}
Any family $\cm=(M_1, \ldots, M_{2n-1})$  of matchings of size $n$
in a bipartite graph possesses a rainbow matching of size $n$.
\end{theorem}

Theorem \ref{originaldrisko} follows upon noting that in a matrix
$A$ as in the theorem every column   defines a matching between rows
and symbols. Theorem \ref{drisko} is sharp, as shown by the
following example:

\begin{example}\label{only}
For $n>1$ let $M_1, \ldots, M_{n-1}$ be all equal to the matching
consisting of all even edges in a cycle of length $2n$, and  $M_n,
\ldots, M_{2n-2}$ be all equal to the matching consisting of  all
odd edges in this cycle. Together these are  $2n-2$ matchings of
size $n$ not having a rainbow matching of size $n$.
\end{example}

This example can be formulated in matrix form, which proves that
also Theorem \ref{originaldrisko} is sharp. Drisko (\cite{drisko},
Conjecture 2) conjectured that this is the only extreme example. The
main aim of this paper is to prove  this conjecture, also in the
 wider context of matchings.

\begin{theorem}\label{main}
Let $\cm=(M_1, \ldots, M_{2n-2})$ be a family of matchings of size
$n$ in a bipartite graph.  If $\cm$ does not possess a partial
rainbow matching of size $n$, then $\bigcup \cm$ is the edge set of
a cycle of length $2n$, half of the members of $\cm$ are the
matching consisting of the even edges of this cycle, and the other
half are equal to the matching consisting of the odd edges in the
cycle.
\end{theorem}

Recently, Bar\'at, Gyarfas and Sarkozy \cite{bgs} proved a
generalization of Theorem \ref{drisko}:

\begin{theorem}\label{bgs}
Any $\lfloor \frac{k+2}{k+1} n \rfloor -(k+1)$ matchings of size $n$
in a bipartite graph have a rainbow matching of size $n-k$.
\end{theorem}

In Section \ref{mcpaths} we shall prove a still more general result:

\begin{theorem}\label{moregeneral}
Let $a \le m$, and let $M_1, \ldots ,M_m$ be  matchings ordered so
that $|M_1| \le |M_2| \le \ldots \le |M_m|$.

Suppose that
\begin{equation}\label{inequality}
\sum_{i\le m-a+1}(|M_i| - a+1) \ge a
\end{equation}
Then there exists a rainbow matching of size $a$.
\end{theorem}

Note that Theorem \ref{drisko} is the case $|M_i|=n$, $m=2n-1$.

\section{Multicolored paths}\label{mcpaths}
The proof of Theorem \ref{main} is based on a simple fact (Theorem
\ref{counting} below) on multicolored paths in networks. We shall
use the following notation for paths in a directed graph.  The
vertex set of a path P is denoted by $V(P)$ and its edge set by
$E(P)$. If $\cp$ is a set of paths, we write  $E[\cp]$ for
$\bigcup_{P \in \cp}E(P)$. If $P$ is a path and $v \in V(P)$ we
write $vP$ for the part of $P$ between $v$ and the last vertex of
$P$. Similarly, $Pv$ is the part of $P$ ending at $v$. If $P,Q$ are
paths,
 $u \in V(P)$ and $v \in V(Q)$ and $uv$ is an edge,  we write $PuvQ$
 for the trail  $Pu$ concatenated with the edge $uv$ and then with $vQ$
 (a {\em trail} is a path with repetition of vertices allowed). A path starting at a vertex $a$ and ending at a vertex $b$ is called an $a-b$ path.

A {\em network} $\cn$ is a triple $(D,s,t)$, where $D=(V,E)$ is a
directed graph, and $s$ (the ``source'') and $t$ (the ``sink'') are
two distinguished vertices in $V$. We assume that no edge goes into
$s$ and  no edge leaves $t$. We write $\vc$ for $V \setminus
\{s,t\}$. Given an $s-t$ path $P$, we write $\vc(P)$ for $V(P)
\setminus \{s,t\}$. Two $s-t$ paths $P$ and $Q$ are said to be {\em
innerly disjoint} if $\vc(P) \cap \vc(Q) =\emptyset$.

 Let $\cp$ be a family (that is, a multiset) of $s-t$ paths in $D$, and suppose that
$\cp= \bigcup_{i \le m}\cp_i$, where each $\cp_i$ is a set of
pairwise innerly disjoint paths. Let  $\cl=(\cp_1, \ldots ,\cp_m)$.
A path $Q$ in $D$ is said to be {\em $\cl$-multicolored}  if its
edge set is a rainbow set for the family $(E[\cp_1],E[\cp_2], \ldots
,E[\cp_m])$, namely if each of its edges belongs to a path from a
different $\cp_i$.
 A vertex $v \in V$ is said to be {\em reachable} if there is an $s-v$
$\cl$-multicolored path. Denote the set of reachable points by
$R(\cl)$.

\begin{theorem}\label{counting}
$|R(\cl)|>|\cp|$.
\end{theorem}

\begin{proof}
By induction on $|\cp|$. Since $R(\emptyset)=\{s\}$, the theorem is
true for $|\cp|=0$. Suppose that the result is true for any network
and  all families of paths with fewer than $|\cp|$ paths.
 Let $X$ be the set of all vertices $x$ such that $sx \in E(P)$ for some $P \in \cp_1$. Contract $s$
and $X$ to a single vertex $s'$ (so, if there are $r$ paths in
$\cp_1$, $r+1$ vertices are contracted to the single vertex $s'$).
For each $P \in \cp$
 let $v$ be the last vertex on $P$ belonging to the
set $\{s\}\cup X$, and let $P'=vP$. Then $P'$ is an $s'-t$ path. Let
$\cp'_i=\{P' \mid P \in \cp_i\}$, and let $\cl'=(\cp'_1, \ldots
,\cp'_m)$. Clearly, the paths in each $\cp'_i$ are pairwise innerly
disjoint.
 By the induction hypothesis
$|R(\cl')|>\sum_{2 \le j \le m} |\cp_j|$. We claim that $R(\cl')
\setminus \{s'\} \subseteq R(\cl)$. To show this, let $y \in
R(\cl')\setminus\{s'\}$, and let $Q$ be an $\cl'$-multicolored $s'-y$
path. Let $s'z$ be the first edge on $Q$, and assume that
it belongs to some path $P'$, where $P \in \cp_i$ for some $i \ge
2$. By the definition of the contracted vertex $s'$, this means that
either $sz \in E(P)$ or $xz \in E(P)$ for some $x \in X$. In the
first case the $s-t$ path $Q'$ (with $s$ replacing $s'$ on $Q'$) is
$\cl$-multicolored (not using $\cp_1$). In the second case the $s-t$
path $sxQ'$ (with $x$ replacing $s'$ on $Q'$) is $\cl$-multicolored,
with $sx$ colored by $\cp_1$.

Since all vertices in $X$ are $\cl$-reachable, it follows that
$|R(\cl)|\ge |R(\cl')|+|X|>\sum_{2 \le j \le m} |\cp_j| +
|X|=|\cp|$,  completing the proof.
\end{proof}

\begin{corollary}\label{st}
If~ $|\cp|>|\vc|$ then $t \in R(\cl)$.
\end{corollary}

Given a matching $F$, a path $P$ is said to be $F$-{\em alternating}
if one of each pair of adjacent edges in $P$ belongs to $F$.  An
$F$-alternating path is called {\em augmenting} if its first and
last vertices do not belong to $\bigcup F$. If $A$ is an augmenting
$F$-alternating path then the symmetric difference of $E(A)$ and $F$
is a matching larger than $F$ by $1$. Each connected component of
the union of two matchings is an alternating path with respect to
each of the matchings, and hence we have the following well known
observation:

\begin{lemma}\label{alternating}
If   $G$ and $H$ are matchings in a graph, and $|H|=|G|+q$,  then $H
\cup G$ contains at least $q$ vertex disjoint augmenting
$G$-alternating paths.
\end{lemma}

{\em Proof of Theorem \ref{moregeneral} from Corollary \ref{st}}.
Let $M_1, \ldots ,M_m$ be matchings in a bipartite graph with
sides $A,B$. We shall show that if $F$ is a  rainbow matching of
size $p<a$ then there exists a larger rainbow matching. Let $J$ be
the set of indices $j$ such that $M_j$ is not represented in $F$.
Contract all vertices in $A \setminus \bigcup F$ to a single vertex
$s$, and all vertices in $B \setminus \bigcup F$ to a single vertex
$t$.

For each edge $f \in F$ let $v(f)$ be a vertex in $\cn$. By  Lemma
\ref{alternating} for each matching $M_j$ not represented in $F$
there exists a family $\cp_j$ of  $\max(|M_j|-a+1,0)$ disjoint augmenting
$F$-alternating paths. Let $\cl=(\cp_1, \ldots ,\cp_{m-a+1})$.

Each path $T$ in each $\cp_j$ gives rise to an $s-t$ path $P(T)$ in
$\cn$ obtained by replacing every edge $ab \in E(T)$, where $a \in A
\cap f, ~b \in B \cap g$ (here $f,g \in F$) by the edge $v(f)v(g)$
in $\cn$, every edge $ab \in E(T)$ in which $a \not \in \bigcup F$
and $b \in f \in F$ by the edge $sv(f)$, and every edge $ab \in
E(T)$ in which $b \in B \setminus \bigcup F$ and $a \in f\cap A~, f
\in F$ by $v(f)t$. If $T$ does not meet $F$ at all, namely it
consists of a single edge, then $P(T)=st$.

 By the assumption  \eqref{inequality} and by
Corollary~\ref{st}, there exists an $s-t$ $\cl$-multicolored path,
which translates to a multicolored augmenting $F$-alternating path
$L$. Taking the symmetric difference of $L$ and $F$ yields a rainbow
matching larger than $F$, as desired. \hfill $\square$

Theorem \ref{bgs} is probably not sharp. In fact, it has been
conjectured in \cite{ab} that $n$ matchings of size $n$ in a
bipartite graph have a rainbow matching of size $n-1$.

\section{The extreme case in Corollary \ref{st}}
 The aim of this section is to characterize the examples showing the tightness of
Corollary \ref{st}, namely those sets $\cp$ of $|V|-2$ paths that do not
possess a $\cp$-multicolored $s-t$ path.

\begin{definition}
A system $\cp$  of $s-t$ paths in a network $\cn$ with the notation
above is called {\em regimented} if there exist pairwise innerly
disjoint $s-t$ paths  $P_1, \ldots ,P_k$ and sets of paths $\cp_1,
\ldots ,\cp_k$ such that   $\cp=\cp_1 \cup
\ldots \cup \cp_k$, and each $\cp_j$ is a set of $|E(P_j)|-1$ paths,
all identical to $P_j$.
\end{definition}

Clearly, a regimented set $\cp$ of $s-t$ paths does not have a
$\cp$-multicolored $s-t$ path.

\begin{assertion}\label{coveringall}
If $\cp$ is regimented and $|\cp|=|\vc|$, then $\bigcup \{V(P) \mid P \in \cp\}=V$.
\end{assertion}

\begin{proof} This follows from the fact that if $\cp_1,
\ldots ,\cp_k$ is a regimentation of $\cp$, with corresponding paths $P_i$, then $|\cp_i|=|V(P_i)\setminus \{s,t\}|$ for each $i \le k$. \end{proof}

\begin{theorem}\label{regimented}
If $|\cp|=|\vc|$ and $\cp$ is not regimented, then there exists an
$s-t$ $\cp$-multicolored path.
\end{theorem}

To prove the theorem, let $m=|\vc|$. Assume, for contradiction, that $\cp$ does not have a multicolored
$s-t$ path. We shall show,  by
induction on $|V|$,  that $\cp$ is regimented. Assume that the result is true for networks with fewer vertices.
Let $x$ be a vertex such that $sx \in P_1$. If $x=t$ then the path $P_1=st$ is multicolored, so we may assume that $x \neq t$.
As
in the proof above of Theorem \ref{drisko}, contract $x$ and $s$,
obtaining a network $\cn'$. Like in that proof, for every $P=P_i$ we write $P'-P_i'$ for the $s'-t$ path obtained from $P$ by this contraction.
Let $\cp'=\{P'_i \mid i=2, \ldots, m$. If there exists in $\cp'$ a $\cp'$-multicolored
$s'-t$ path, then as in the above proof it follows that there exists
a $\cp$-multicolored $s-t$ path. So, we may assume that there is no
such path, which, by the induction hypothesis, implies that $\cp'$
is  regimented. Let $Q_1, \ldots ,Q_k$ be the paths regimenting
$\cp'$, and let $\cq_1, \ldots ,\cq_k$ be the corresponding sets of
paths (so  $Q=Q_j$ for every $Q \in \cq_j$).\\

\begin{assertion}\label{assn:1}
$\bigcup_{i=1}^k \vc(Q_i)=\vc\setminus\{x\}$.
\end{assertion}

\begin{proof}
If $|\bigcup_{i=1}^k \vc(Q_i)|<|\vc|-1$ then $|\vc(\cp')|<|\cp'|$ and by Corollary~\ref{st} $\cp'$ has a multicolored path, contradicting our assumption. Hence, $|\bigcup_{i=1}^k \vc(Q_i)|=|\vc|-1$ and the result follows by Assertion \ref{coveringall}.
\end{proof}

\begin{assertion}\label{nojumps}
Let $k\ge 2$. If   $uv \in E(P_k)$ then there exists $i$ such that $\{u,v\} \subseteq V(Q_i)$.
\end{assertion}

\begin{proof}
Assuming negation, by Assertion \ref{coveringall} there exist $i \neq j$ such that $u \in V(Q_i) \setminus V(Q_j)$ and $v \in V(Q_j) \setminus V(Q_i)$.
$P\in \cp$ be a path such $Q_i=P'$. Then the path $PuvQ_j$ is multicolored, where the edge $sx$, if present in $P$, is colored by $P_1$,
the edges in $Q_i$ by paths from $\cq_i$, the edge $uv$ by $P_k$ and the edges of $Q_j$ by paths from $\cq_j$.
\end{proof}

We can now prove Theorem \ref{regimented}.

{\bf Case I:}~~The path $xP_1$, with $x$ replaced by $s'$, is equal to some $Q_i$.

Assume first that all $P \in \cp$ for which $P'=Q_i$ are equal to $sxQ_i$. Let $j \neq i$ and let $P \in \cp$ such that $P' \in \cq_j$. If $P=sxQ_j$, then $P_1$ is multicolored, where $sx$ is colored by $P$ and the remaining
 edges of $P_1$ are colored by paths from $\cq_i$ and $P_1$. So, we may assume that for every $j \neq i$ and every $P \in \cp$ for which $P' \in \cq_j$  we have $P=sQ_j$. But this means that $\cp$ is regimented, the regimenting paths being $sxQ_i$ and  $sQ_j,~j \neq i$.

 Thus we may assume in this case that there exists $P \in \cp$ such that $P'=Q_i$ but $P\ne sxQ_i$. There are two options: (i)~$x\in V(P)$ and the initial path $Px$ contains a vertex different from $s$ and $x$. Let $y$ be the last such vertex on $Px$. By Assertion \ref{coveringall} there exists $j \in \{2, \ldots, m\}$ different from $i$,  such that $y \in V(Q_j)$. But then the edge $yx$ shows, by Assertion \ref{nojumps}, that there exists a multicolored $s-t$ path. (ii) $x\not\in V(P)$, that is $P=sQ_i$. Then $P$ is a multicolored $s-t$, where the first edge is colored by $P$ and the rest are colored by the paths in $\cq_i\setminus\{P\}$ and $P_1$. \\

 {\bf Case II:}~~The path $xP_1$, with $x$ replaced by $s'$, is not equal to any $Q_i$.

 There are two subcases: (i) $V(xP_1\setminus\{x\})\subsetneq V(Q_i\setminus\{s'\})$. This means that there exist two consecutive vertices $u,v\in V(xP_1)$ such that there exist at least one vertex in $V(Q_i)$ between $u$ and $v$. Then $P_1$ is a multicolored $s-t$, where $uv$ is colored by $P_1$ and the rest are colored by the paths $P$ such that $P'\in\cq_i$. (ii) There exist two consecutive vertices $u,v$ on $P_1$ such that the edge $uv$ does not belong to $\bigcup_{2 \le j \le m}E(Q_j)$. By Assertion \ref{coveringall} this means that there exist $i\neq j$ in $\{2, \ldots ,m\}$ such that $u \in V(Q_i)$ and $v \in V(Q_j)$. Then, again by Assertion \ref{nojumps}, there exists a multicolored $s-t$ path.

This concludes the proof of Theorem \ref{regimented}.

\section{Uniqueness of the extreme case in Theorem \ref{drisko}}

In this section we deduce Theorem \ref{main} from Theorem \ref{regimented}.
First, let us prove a weaker version of the theorem:

\begin{lemma}\label{almost}
Under the  conditions of Theorem \ref{main},
there exists a matching of size $n$ representing at least $n-1$ matchings $M_i$.
\end{lemma}

\begin{proof}
Add to the system $M_1, \ldots ,M_{2n-2}$ a matching that is equal to one of the matchings $M_i$. By Theorem \ref{drisko} there exists now a rainbow matching of size $n$, which may represent $M_i$ twice,
represents in all $n-1$ matchings.
\end{proof}

{\em Proof of Theorem \ref{main}}

Let $F=\{e_1, \ldots ,e_{n-1}, e_n\}$ be a matching of size $n$,
representing  $n-1$ matchings among $M_1, \ldots ,M_{2n-2}$, as in Lemma \ref{almost}. Without
loss of generality we may assume that $e_i \in M_{n-1+i}$ for all
$i<n$, and that $e_n \in M_n$. Thus, both $e_1$ and $e_n$ represent
$M_n$, each of $M_{n+1}, \ldots ,M_{2n-2}$ is represented once, and $M_1, \ldots ,M_{n-1}$ are not represented at all.

Let $A,B$ be the sides of the bipartite graph,  and let $e_i=a_ib_i$,
where $a_i \in A, ~b_i\in B$. Assign to every edge $e_i$, $i <n$, a
vertex $v_i$, and let $V=\{s,t,v_1, \ldots ,v_{n-1}\}$. An augmenting $F$-alternating path $Q$ corresponds then to an $s-t$ path $P=P(Q)$ in a network $\cn$ on $V$, as follows.
The first edge of $Q$, which is $ub_i$ for some $u \in A \setminus \bigcup F$, is assigned the edge $sv_i$ in $E(P)$. Every edge $a_ib_j \in E(Q)$, where $a_i,~b_j \in \bigcup F$,  is assigned the edge $v_iv_j \in E(P)$.
The last edge of $Q$, which is $a_iw$ for some $w \in B \setminus \bigcup F$, is assigned the edge $v_it$ in $E(P)$.

As in the proof of Theorem \ref{drisko}, each matching $M_i, i <n$
generates an $F$-alternating path $Q_i$, which then translates to an $s-t$
path $P_i=P(Q_i)$ in $\cn$. By Theorem \ref{regimented} we may assume that
the system $\cp=(P_1, \ldots ,P_{n-1})$ is regimented by a set $\cp_j, j \in J$ of paths, where all paths in $\cp_j$ are equal
to the same path $P_j$.

\begin{assertion}\label{noinbetween}
No two distinct paths in $\cp$ are connected by an edge belonging to some $E(Q_i),~ i<n$.
\end{assertion}

\begin{proof}
Suppose that $\ell_1  \neq \ell_2$, and that $a_jb_k \in E(Q_i)$ for some $v_j \in V(P_{\ell_1}),~v_k\in V(P_{\ell_2})$. Clearly, then, $i \not \in \{\ell_1, \ell_2\}$, and $v_j \neq t, ~v_k \neq s$. These imply that $P_{\ell_1}v_jv_kP_{\ell_2}$ is  an $\cm$-multicolored $s-t$ path:
 there are enough paths in $\cp_{\ell_1}$ to color $E(P_{\ell_1}v_j)$, enough paths in $\cp_{\ell_2}$ to color $v_kP_{\ell_2}$, and the edge $v_jv_k$ is colored by $P_i$.
\end{proof}

 \begin{assertion}\label{loops}
 For every $i<n$ if $v_j \in V \setminus V(P_i)$ then  $e_j \in M_i$.
 \end{assertion}

 \begin{proof}
  Let $E_i=\{v_jv_\ell \mid a_jb_\ell \in M_i\}$.
    Since $\bigcup V(P_i)=V$, and since $E_i$ is the union of $V(P_i)$ and cycles, if $(v) \not \in E_i$ then  $v$ lies on a non-singleton cycle $C$ in $E_i$. By Assertion \ref{noinbetween}
  $V(C) \subseteq V(P_j)$ for some $ j \neq i$. Let $\ell$ be such that $P_j \in \cp_\ell$. There exists an edge in $C$, say $xy$, that goes forward on $P_j$. Then the path $P_jxyP_j$ is $\cm$-multicolored, where $xy$ is colored by $M_i$ and $E(P_jx) \cup E(yP_j)$ are colored by
 matchings $M_k$ for which $P_k \in \cp_\ell$.
   \end{proof}

 \begin{assertion}
 $|J|=1$.
 \end{assertion}

 \begin{proof}
 Suppose not. Then there exists $i<n$ such that $v_1 \not \in V(P_i)$, and by Assertion \ref{loops} $e_1 \in M_i$. Re-coloring $e_1$ by $M_i$, and keeping the coloring of all other $e_i$'s
 (that is, $e_i$ is colored  by $M_{n-1+i}$ for all $i\neq 1$) results then in a rainbow matching of size $n$.
   \end{proof}

Let the single path in the
regimentation be $P=sv_{i_1}v_{i_2}\ldots v_{i_{n-1}}t$. For each $k < n$ the fact that the first edge of $P$, ~~$sv_{i_1}$, belongs to
$P_k$ means that there exits $c(k) \in A\setminus \{a_1, \ldots
,a_{n-1}\}$ such that $c(k)b_{i_1} \in P_k$.

\begin{assertion}
$c(k) = a_n$.
 \end{assertion}

(Remember that $e_n=a_nb_n$ is one of
the two edges of $F$ belonging to $M_n$)
\begin{proof}
Assume that $c(k) \neq a_n$. Let $j$ be such that $i_j=1$.
 Apply to $F$ the alternating path $P_kb_{i_j}$ (namely the initial part of $P_k$, up to and including $b_{i_j}$), and add the edge $e_n$. The resulting matching, $F \triangle E(P_kb_{i_j}) \cup \{e_n\}$, can be colored by $n$ matchings $M_i$
($e_n$ is colored by color $1$), contradicting the negation assumption.
\end{proof}

Symmetrically,  the edge $v_{i_{n-1}}t$ represents the edge
$a_{i_{n-1}}b_n$ in all $P_k$s. This means that all  matchings $M_k$, ~~ $k
\le n-1$, are identical, each consisting of the edges $a_nb_{i_1},
a_{i_1}b_{i_2}, \ldots ,a_{i_{n-1}}b_n$.

In particular, this easily implies that there exists a  matching of size $n$ representing each of the matchings $M_1, \ldots ,M_{n-1}$, one of them twice. Applying a symmetric argument to the one above, we deduce that also $M_n, \ldots, M_{2n-2}$ are identical matchings, and this clearly implies that they are all equal to the matching $F$ in the proof above. This shows that the matchings $M_1, \ldots M_{2n-2}$ form a cycle, with the even edges being in the first $n$ matchings and the odd edges being the last $n-1$ matchings, as desired.

\section{The extreme case in the Erd\H{o}s-Ginzburg-Ziv theorem}

The Erd\H{o}s-Ginzburg-Ziv theorem \cite{egz} states that a multiset of $2n-1$ elements in $\mathbb{Z}_n$ has a sub-multiset of size $n$ summing up to $0 \pmod n$.  Alon \cite{alon} noted that the EGZ theorem can be deduced from Theorem \ref{drisko}, as follows. Let $A$ be a multi set of
size $2n-1$ in $\zn$. For each $a \in A$ define a matching $M_a$ in a bipartite graph with both sides indexed by $\zn$, by $M_a=\{(i,i+a) \mid i \in \zn\}$. By Theorem \ref{drisko} there exists a sub-multiset $B$ of size
$n$ of $A$, and a matching $(i,i+b(i)) \mid i \in \zn\}$, where $~~b(i) \in B$. Then the sum $\sum_{i \in \zn} i+b(i)$ includes in its terms every element of $\zn$ precisely once, and hence $\sum_{i \in \zn} i+b(i)\equiv\sum_{i \in \zn} i \pmod n$, implying that $\sum_{i \in \zn} b(i)\equiv 0 \pmod n$.

An example showing tightness is a multiset consisting of $n-1$ copies of $a$ and $n-1$ copies of $b$, where $gcd(b-a,n)=1$. The following was proved in  \cite{abc, flores, yuster}, and here we give it yet another proof.

\begin{theorem}
Any multiset  of size $2n-2$ in $\zn$ not having a sub-multiset of size $n$ summing up to $0 \bmod n$ is of the above form.
\end{theorem}

\begin{proof}
Let $A$ be a multiset with the above property. As shown above, if there is no multi-subset as desired then the set of matchings $M_c=\{(i,i+c) \mid i \in \zn\}$, $c\in A$, does not have a rainbow matching of size $n$. By Theorem \ref{main} it follows that there are $a,b$ such that $n-1$ of the matchings $M_c$, $c\in A$, are equal to $M_a$, and $n-1$ of them are equal to $M_b$. It remains to show that $gcd(b-a,n)=1$. If $gcd(b-a,n)>1$ then there exists $k$ with $0< k<n$ such that $k(b-a) \equiv 0 \pmod n$.
Then the sum of $k$ copies of $b$ and $n-k$ copies of $a$ is $0 \pmod n$, contrary to the assumption on $A$.
\end{proof}

{\bf Acknowledgement}
We are indebted to Eli Berger for a helpful observation. 



\begin{thebibliography}{99}
\bibitem{ab} R. Aharoni and E. Berger, rainbow matchings in $r$-partite hypergraphs, Electronic J. of Comb., {\bf 16}, Issue 1 (2009).

\bibitem{alon}
 N. Alon, Multicolored matchings in hypergraphs, {\em Moscow Journal of Combinatorics
and Number Theory}, \textbf{1} (2011), 3--10.


\bibitem {abc} N. Alon, A. Bialostocki and Y. Caro, The extremal cases in Erd{\"o}s-Ginzburg-Ziv theorem, unpublished (1991)
\bibitem{drisko}
A. A. Drisko,
Transversals in row-Latin rectangles, {\em J. Combin. Theory Ser.
A}, \textbf{84} (1998), 181--195.


\bibitem{bgs} J. Bar\'at, A. Gyarfas and Sarkozy, Rainbow matchings in bipartite multigraphs,
arXiv:1505.01779v2.
\bibitem{egz}
P. Erd\H{o}s, A. Ginzburg and A. Ziv,
A theorem in additive number theory, {\em Bull. Res. Council Israel 10F}, (1961), 41--43.


\bibitem{flores} C. Flores and O. Ordaz, On the Erd{\"o}s-Ginzburg-Ziv theorem,  {\em Discrete Mathematics}
, \textbf{152} (1) (1996), 321--324.
\bibitem{stein} S. K. Stein, Transversals of Latin squares and their generalizations, {\em Pacific J. Math.} {\bf 59} (1975), 567--575.


\bibitem{yuster}
T. Yuster,
Bounds for counter-example to addition theorem in solvable groups
{\em Arch. Math. (Basel)}, \textbf{51} (1988), 223--231.

\end{thebibliography}
\end{document}